\documentclass[12pt]{amsart}
\usepackage{amssymb}
\usepackage{eucal}
\usepackage{mathrsfs}
\usepackage{a4wide}

\newtheorem{theorem}{Theorem}[section]
\newtheorem{corollary}[theorem]{Corollary}
\newtheorem{lemma}[theorem]{Lemma}
\newtheorem{proposition}[theorem]{Proposition}
\theoremstyle{remark}
\newtheorem{claim}{Claim}
\newtheorem{question}[theorem]{Question}

\makeatletter
\def\identity{\ensuremath{{\mathop{\operator@font id}\nolimits}}}
\makeatother

\begin{document}

\title{A universal $P$-group of weight $\aleph$}
\author{Jan van Mill}

\address{KdV Institute for Mathematics,
University of Amsterdam,
Science Park 105-107,
P.O. Box 94248,
1090 GE Amsterdam, The Netherlands}
\email{j.vanMill@uva.nl}

\date{\today}



\begin{abstract}
We show that under the Continuum Hypothesis, the topological group of all homeomorphisms of the \v{C}ech-Stone remainder of $\omega$ with the $G_\delta$-topology, is a universal object for all $P$-groups of weight at most ${\mathfrak c}$.
\end{abstract}

\maketitle

\section{Introduction}\label{introduction}
\emph{All topological spaces under discussion are Tychonoff}

Uspenskiy~\cite{Uspenskii86} proved in 1986 that there is a universal topological group $G$ with a countable base. Universal in the sense that every topological group with a countable base is (topologically isomorphic to) a subgroup of $G$. It is the homeomorphism group $\mathcal{H}(Q)$ of the Hilbert cube $Q=\mathbb{I}^\omega$, endowed with the compact-open topology. (Another such group is the group of isometries of the Urysohn universal space, with the topology of pointwise convergence, \cite{Uspenskii90}.) The obvious question from \cite{Uspenskii90} -- whether there are universal topological groups of uncountable weight -- is still open. Natural candidates are the homeomorphism groups of the Tychonoff cubes $\mathbb{I}^\tau$, where $\tau$ is uncountable. Uspenskiy's proof in \cite{Uspenskii86} is based on Keller's Theorem~\cite{keller} from 1931 that all compact convex and infinite-dimensional subsets of Banach spaces, are homeomorphic to $Q$. There are good topological characterizations of the Tychonoff cubes $\mathbb{I}^\tau$ for uncountable $\tau$ by \u S\u cepin~\cite{scepin79} (see also \cite[7.2.9]{chico2}). However, a necessary condition in these characterizations is that the spaces under consideration are Absolute Retracts, which is problematic for arbitrary compact convex subsets of $\mathbb{I}^\tau$. (There are even compact convex subsets of $\mathbb{I}^{\omega_1}$ that do not satisfy the countable chain condition.) This blocks the attempt to prove that the spaces $\mathcal{H}(\mathbb{I}^\tau)$ are universal for all uncountable $\tau$ by a generalization of the arguments in Uspenskiy~\cite{Uspenskii86}.

Shkarin~\cite{shkarin99} proved that there are universal Abelian topological groups with a countable base. And if $\tau$ is an infinite cardinal such that $2^\lambda \le \tau$ for every cardinal
$\lambda < \tau$, then there is a group that is universal in the class of all Abelian topological groups of weight at most~$\tau$.

As usual, we denote the \v{C}ech-Stone compactification of the discrete space $\omega$ by $\beta\omega$, and $\omega^* = \beta\omega\setminus\omega$, its \v{C}ech-Stone remainder. A compact, zero-dimensional $F$-space of weight ${\mathfrak c}$ in which each non-empty $G_\delta$-subset
has infinite interior will be called a \emph{Parovi\v{c}enko space}. Here an $F$-space is a space in which cozero-subsets are $C^*$-embedded. For normal spaces, this simply boils down to the statement that disjoint open $F_\sigma$-subsets have disjoint closures. Under the Continuum Hypothesis ({\ensuremath{\mathsf{CH}}}), all Parovi\v{c}enko spaces are homeomorphic to $\omega^*$, as was shown by Parovi\v{c}enko~\cite{paro:universal}. Subsequently, van Douwen and van Mill~\cite{vm:14} obtained the converse: Parovi\v{c}enko's characterization of $\omega^*$ implies {\ensuremath{\mathsf{CH}}}.
Parovi\v{c}enko also showed that under {\ensuremath{\mathsf{CH}}}, every compact space of weight at most ${\mathfrak c}$ is a continuous image of $\omega^*$, hence $\omega^*$ is \emph{universal} in the sense of ‘mapping onto’. Dow and Hart~\cite{DowHart00} proved that under {\ensuremath{\mathsf{CH}}}, the
\v{C}ech-Stone remainder $[0,\infty)^*$ of the half line $[0,\infty)$ is universal among the continua of weight ${\mathfrak c}$, again in the sense of ‘mapping onto’. (It is an open problem whether in some model there is a topological characterization of $[0,\infty)^*$ in the style of Parovi\v{c}enko.)

We are interested in $\mathcal{H}(\omega^*)$, the homeomorphism group of $\omega^*$ with the compact-open topology. Like $\omega^*$ itself, it is a mysterious object. Rudin~\cite{Rudin56} proved that under {\ensuremath{\mathsf{CH}}}, its cardinality is $2^{\mathfrak c}$. And, under the same condition, it is simple by Fuchino~\cite{fuchinothesis}. It even remains simple
if $\aleph_2$ Cohen reals are added to a model of {\ensuremath{\mathsf{ZFC}}\ + {\ensuremath{\mathsf{CH}}}, \cite{fuchino92}. But Shelah~\cite{shelah:properforcing} proved that it is consistent that all homeomorphisms of $\omega^*$ are `trivial' in the sense that they are induced by bijections between cofinite subsets of $\omega$; hence, in this model, $\mathcal{H}(\omega^*)$ has cardinality ${\mathfrak c}$ in contrast to Rudin's result, and is by van Douwen~\cite{simple} not simple, in contrast to Fuchino's result.

The goal of this note is to prove that $\mathcal{H}(\omega^*)$ exhibits similar behavior within the class of suitable topological groups as $\omega^*$ does within the class of suitable topological spaces.

Following Arhangel'skii and Tkachenko~\cite{ArhangTkachenko08},  we will call a topological group whose underlying topological space is a $P$-space -- that is a space in which every $G_\delta$-subset is open -- a \emph{$P$-group}. If $G$ is a topological group, then $G_{(\delta)}$ denotes its \emph{$G_\delta$-modification}. The underlying set of $G_{(\delta)}$ is $G$, and the $G_\delta$-subsets of $G$ form a basis (for the open subsets) of $G_{(\delta)}$. It is easy to see that $G_{(\delta)}$ is a topological group, \cite[3.6.H]{ArhangTkachenko08}, and a $P$-space, hence is a $P$-group.

\begin{theorem}\label{eerstestelling}
 Let $G$ be a topological group of weight ${\mathfrak c}$. Then there is a Parovi\v{c}enko space~$X$ such that $G_{(\delta)}$ is (topologically isomorphic to) a subgroup of $\mathcal{H}(X)$.
\end{theorem}

This is in line with van Douwen's unpublished result that every $P$-space of weight at most ${\mathfrak c}$ embeds in $\omega^*$; see \cite{vm:70} for details and references.

\begin{theorem}[{\ensuremath{\mathsf{CH}}}]\label{eerstecor}
$\mathcal{H}(\omega^*)_{(\delta)}$ is universal for the class of all $P$-groups of weight at most~${\mathfrak c}$.
\end{theorem}

The following is in line with the results in \cite{paro:universal, vm:14} that we discussed above.

\begin{theorem}\label{tweedestelling}
{\ensuremath{\mathsf{CH}}}\ is equivalent to the statement that for all Parovi\v{c}enko spaces $X$ and $Y$, $\mathcal{H}(X)$ and $\mathcal{H}(Y)$ are (algebraically) isomorphic.
\end{theorem}

We also discuss some topological properties of the group $\mathcal{H}(\omega^*)$. We show, among other things, that {\ensuremath{\mathcal{H}(\omega^*)}}\ is not a $P$-group but that all of its countable subsets are closed and discrete.

The class of `small' $P$-groups is rather limited, so our results on universality are not a strong addition to what is known. But they may indicate that the homeomorphism group of $\omega^*$ is by itself an interesting object that is worth studying.

\section{Preliminaries}\label{preliminaries}
\subsection{Topology} For any space $X$, we let $\tau X$ denote its topology, and $\mathcal{H}(X)$ its group of homeomorphisms. We let $e_{\mathcal{H}(X)}$ denote the identity element of the group $\mathcal{H}(X)$, that is, the identity function $\identity_X$ of $X$. If no confusion is likely, we suppress the index. We use standard terminology for cardinal functions. For example, if $X$ is a space, then $w(X)$ and $\chi(X)$ denote its weight and character, respectively (see Juh\'asz~\cite{juhasz}).

If $X$ is a space, then $\beta X$ and $X^*=\beta X\setminus X$ denote its \v{C}ech-Stone compactification and \v{C}ech-Stone remainder, respectively. If $X$ is normal, then disjoint closed subsets of $X$ have disjoint closures in $\beta X$, \cite[3.6.4]{engelking:gentop}. In fact, this property characterizes $\beta X$.

 It is well-known, and easy to prove, that the \v{C}ech-Stone compactification $\beta X$ of $X$ is zero-dimensional iff $X$ is strongly zero-dimensional; that is, disjoint zero-sets in $X$ can be separated by disjoint clopen subsets of $X$, \cite[6.2.12]{engelking:gentop}.

If $C\subseteq X$ is clopen, then $C^* = (\mathrm{cl}_{\beta X} C)\cap X^*$. It is easy to see that if $C$ and $D$ are clopen in $X$ and $C\setminus D$ is compact, then $C^*\subseteq D^*$. Hence if $C\triangle D$ is compact, then $C^*=D^*$; here $\triangle$ denotes symmetric difference of course. Moreover, if $X$ is locally compact and strongly zero-dimensional, and $E\subseteq X^*$ is (relatively) clopen, then there is a clopen subset $C$ of $X$ such that $C^*=E$. This easily implies the well-known fact that all nonempty (relatively) clopen subsets of $\omega^*$ are homeomorphic to $\omega^*$. This will be used frequently in the remaining part of this note.

\subsection{Groups}
 If $G$ and $H$ are groups, then by an \emph{isomorphism} from $G$ to $H$ we will mean a group isomorphism (no topology involved). If $G$ and $H$ are topological groups, then a \emph{topological isomorphism} between $G$ and $H$ is an isomorphism that is also a homeomorphism.

Let $G$ be a topological group with subgroup $H$. It is easily seen that $H_{(\delta)}$ is (topologically isomorphic to) the subgroup $H$ of $G_{(\delta)}$.

If $\mathcal{K}$ is a class of topological groups, then $G\in \mathcal{K}$ is said to be \emph{universal} (for $\mathcal{K}$) provided that every member of $\mathcal{K}$ is (topologically isomorphic to) a subgroup of $G$.

\subsection{The topological groups $\mathcal{H}(X)$ for compact zero-dimensional $X$}\label{prelimeen}
Let $X$ be a compact space. We endow $\mathcal{H}(X)$ with the standard \emph{compact-open} topology. That is the topology for which the collection of all $[K,U]$'s is a subbase for the open sets; here $K$ is compact, $U\in\tau X$, and
$$
    [K,U] = \{f\in \mathcal{H}(X) : f(K)\subseteq U\}. \leqno{(\dag)}
$$
It is well-known that by compactness of $X$, this topology makes $\mathcal{H}(X)$ a topological group. In fact, the natural action $\mathcal{H}(X)\times X\to X$ defined by $(g,x)\mapsto g(x)$, $x\in X$, is a (continuous) action of the topological group $\mathcal{H}(X)$ on $X$.

If $X$ is compact and zero-dimensional, then we can simplify $(\dag)$, as follows. If $f\in [K,U]$, then we may pick a clopen subset $F$ of $X$ such that $K\subseteq F \subseteq f^{-1}(U)$. Hence $f\in [F,G]\subseteq [K,U]$, where both $F$ and $G = f(F)$ are clopen. This implies that $w(\mathcal{H}(X))\le w(X)$. For $C\subseteq X$ clopen, we put
$$
    \hat C= [C,C] \cap [X\setminus C,X\setminus C]= \{f\in \mathcal{H}(X) : f(C) = C\}.
$$
Clearly, $\hat C$ is an open subgroup of $\mathcal{H}(X)$, hence even a clopen subgroup. A moments reflection shows that if $\mathcal{E}$ is a finite clopen partition of $X$, then $N(\mathcal{E})= \{f\in \mathcal{H}(X): (\forall\, E\in \mathcal{E})(f(E)=E)\}$ is a clopen subgroup of $\mathcal{H}(X)$, and that the collection of all $N(\mathcal{E})$'s is a neighborhood base at $e$ in $\mathcal{H}(X)$.
This is well-known, see e.g., \cite[3.2]{MegrelishviliScarr01}. (There is a typo in the proof, though: op page 272, line 13. \emph{`base'} should be changed to \emph{`subbase'}.) 

A topological group is \emph{non-archimedean} if it has a local base at the identity consisting of open (hence clopen) subgroups. Hence $\mathcal{H}(X)$ for compact zero-dimensional $X$ is non-archimedean. It is easy to see that $P$-groups are non-archimedean as well, \cite[4.4.1]{ArhangTkachenko08}.


\subsection{Teleman's Theorem and generalizations}

In \cite{teleman57}, Teleman showed (among other things) that every topological group is topologically isomorphic to a subgroup of $\mathcal{H}(X)$, for some compact space $X$.  We need a version of his result for compact zero-dimensional spaces. Here are the results that we need:

\begin{theorem}[{Megrelishvili and Shlossberg~\cite[3.2]{MegrelishviliShlossberg12}}]\label{megresh} Let $G$ be a topological group.
The following statements are equivalent:
\begin{enumerate}
\item $G$ is non-archimedean.
\item $G$ is (topologically isomorphic to) a topological subgroup of $\mathcal{H}(X)$, for some compact zero-dimensional space $X$ with $w(X) = w(G)$.
\end{enumerate}
\end{theorem}


\subsection{Rubin's Theorem}
 A subset of a topological space is \emph{somewhere dense} if its closure has nonempty interior.

 Let $X$ be a space, and $G$ a subgroup of $\mathcal{H}(X)$. We say that the pair $\\langle{X},{G}\rangle$ satisfies \emph{Rubin's condition $\mathfrak{R}$}, abbreviated $\langle{X},{G}\rangle\in \mathfrak{R}$, if the following condition holds:

 \begin{enumerate}
 \item[$(\mathfrak{R})$]  for every
open $U \subseteq X$ and $x \in U$, $\{g(x) :g\in G$ and $g{\restriction}(X \setminus U) = \identity_{X\setminus U}\}$ is somewhere dense.
 \end{enumerate}

 We will need the following deep result:

 \begin{theorem}[{Rubin~\cite[Corollary 3.5(c)]{RubinM89}}]\label{rubin}
If $\langle{X},{G}\rangle, \langle{Y},{H}\rangle\in \mathfrak{R}$ and both $X$ and $Y$ are locally compact, then for each isomorphism $\varphi\colon G\to H$, there exists a homeomorphism $f\colon X\to Y$ such that for all $g\in G$, $f \circ g= \varphi(g)\circ f$.
 \end{theorem}

 We will us this result in the following case. The space $X$ is locally compact and zero-dimensional, has a clopen base $\mathcal{B}$ of pairwise homeomorphic sets, and $G=\mathcal{H}(X)$. Hence the elements of $\mathcal{B}$ are compact. We claim that $\langle{X},{\mathcal{H}(X)}\rangle\in \mathfrak{R}$. Indeed, if $U\subseteq X$ is open, then for each $x\in U$, $O(x)=\{g(x) :g\in \mathcal{H}(X)$ and $g{\restriction}(X \setminus U) = \identity_{X\setminus U}\}$, is dense in $U$. For take an arbitrary nonempty $B_1\in \mathcal{B}$ such that $B_1\subseteq U$. We may assume that $x\not\in B_1$. Take $B_0\in \mathcal{B}$ such that $x\in B_0\subseteq U\setminus B_1$. Let $\xi\colon B_0\to B_1$ be any homeomorphism, and define $f\colon X\to X$ by
 $$
    f(y) = \begin{cases}
            \xi(y) & (y\in B_0),\\
            \xi^{-1}(y) & (y\in B_1),\\
            y     & (y\in X\setminus (B_0\cup B_1).
           \end{cases}
 $$
Then $f\in \mathcal{H}(X)$, $f{\restriction}(X \setminus U) = \identity_{X\setminus U}$,  and $f(x)\in B_1\cap O(x)$.

We impose similar conditions on $Y$.

\section{Proofs of Theorems \ref{eerstestellingzonder} and \ref{eerstecorzonder}}

Let $X$ be a fixed nonempty compact and zero-dimensional space. Put $Z= \omega\times X$; here $\omega$ has the discrete topology. Then $Z$ is locally compact, and so $Z^*$ is compact. If $X$ has weight at most ${\mathfrak c}$, then $Z$ has at most ${\mathfrak c}^\omega={\mathfrak c}$ clopen sets, so $Z^*$ is a Parovi\v{c}enko space by \cite[1.2.5]{vanmill:betaomega}. Therefore, $Z^*$ and $\omega^*$ are homeomorphic under {\ensuremath{\mathsf{CH}}}. The interplay between $\mathcal{H}(X)$ and $\mathcal{H}(Z^*)$ will be used here for proving \ref{eerstestelling}.

Since $Z$ is strongly zero-dimensional, each (relatively) clopen subset of $Z^*$ has the form $F^*$, where $F\subseteq \omega\times X$ is clopen. So a standard subbasic neighborhood of some $g\in \mathcal{H}(Z^*)$ has, by what we observed in \S\ref{prelimeen}, the form $[F^*,G^*]$, where $F$ and $G$ are arbitrary clopen subsets of $Z$.

If $f\in \mathcal{H}(X)$, we define $f^\beta\in \mathcal{H}(Z^*)$, as follows. We first let $f^\omega$ denote the homeomorphism of $Z$ defined by $f^\omega(\langle{n},{x}\rangle) = \langle{n},{f(x)}\rangle$ $(n < \omega, x\in X)$. Then we let $\beta f^\omega\colon \beta Z \to \beta Z$ be its Stone extension. Finally, let $f^\beta = \beta f^\omega\restriction  Z^*$. Clearly, $f^\beta\in \mathcal{H}( Z^*)$. Define $i\colon \mathcal{H}(X)\to \mathcal{H}( Z^*)$ by $i(f) = f^\beta$.

\begin{lemma}\label{eerstelemma}
$i$ is an injective homomorphism.
\end{lemma}

\begin{proof}
That $i$ is a homomorphism is clear.

To check that it is injective, take distinct $f,g\in \mathcal{H}(X)$. There exists $x\in X$ such that $f(x)\not= g(x)$. Then $S= \omega\times \{f(x)\}$ and $T= \omega\times \{g(x)\}$ are disjoint closed subsets of $Z$, hence, by normality of $Z$, have disjoint closures in $\beta Z$. Pick an arbitrary $p\in \mathrm{cl}_{\beta Z} (\omega\times \{x\})\setminus (\omega\times \{x\})$. Then $\beta f^\omega(p)\in \mathrm{cl}_{\beta X}(S)$ and  $\beta g^\omega(p)\in \mathrm{cl}_{\beta X}(T)$, and so $f^\beta(p)\not= g^\beta(p)$.
\end{proof}

\begin{lemma}\label{zaterdageen}
Let $F$ and $G$ be clopen in $Z$ and $f\in \mathcal{H}(X)$ such that $f^\beta\in [F^*,G^*]$. Then
$\{ n < \omega : f^\omega(F\cap (\{n\}\times X))\not\subseteq G\cap (\{n\}\times X)\}$ is finite.
\end{lemma}

\begin{proof}
If not, then we may pick an infinite subset $A$ of $\omega$ and for each $n\in A$ an element $x_n\in X$ such that $\langle{n},{x_n}\rangle\in F\cap (\{n\}\times X)$ but $f^\omega(\langle{n},{x_n}\rangle)\not\in G\cap (\{n\}\times X)$. Let $p$ be a limit point of $\{\langle{n},{x_n}\rangle: n\in A\}$ in $Z^*$. Then, clearly, $p\in F^*$. Let $q= f^\beta(p)$. Then $q\in Z^*$ and hence by our assumptions, $q\in G^*$. Moreover, it is a limit point of $L=\{\langle{n},{f(x_n)\rangle} : n\in A\}$. Now $L$ is a closed subset of $Z$ which is disjoint from $G$. Hence, by normality of $Z$, $\mathrm{cl}_{\beta X} L \cap \mathrm{cl}_{\beta X} G = \emptyset$. So $q\not\in G^*$, which is a contradiction.
\end{proof}

\begin{lemma}\label{tweedelemma} ${}$
$i\colon \mathcal{H}(X)_{(\delta)}\to \mathcal{H}(Z^*)$ is continuous.
\end{lemma}

\begin{proof}
Let $F$ and $G$ be clopen in $Z$, and assume that $f^\beta\in [F^*,G^*]$.
By \ref{zaterdageen}, there exists $N< \omega$ such that for each $n \ge N$, $f^\omega(F\cap (\{n\}\times X))\subseteq G\cap (\{n\}\times X)\}$. For each $n\ge N$, take clopen subsets $P_n$ and $Q_n$ of $X$ such that $\{n\}\times P_n = F\cap (\{n\}\times X)$ and $\{n\}\times Q_n = G\cap (\{n\}\times X)$, respectively. Hence for each $n \ge N$, $f\in [P_n,Q_n]$. Put $S=\bigcap_{n\ge N}[P_n,Q_n]$. Then $S$ is a neighborhood of $f$ in $\mathcal{H}(X)_{(\delta)}$. Now assume that $g\in S$. Then
$$
    g^\omega(\textstyle \bigcup_{n\ge N} F\cap(\{n\}\times X)) = g^\omega(\textstyle \bigcup_{n\ge N} \{n\}\times P_n)\subseteq \bigcup_{n\ge N} \{n\}\times Q_n = \bigcup_{n\ge N} G\cap(\{n\}\times X),
$$
and so $g^\beta(F^*)\subseteq G^*$. To see that this is true, simply observe that $\beta g^\omega$ is a homeomorphism of $\beta X$ and that
$$
    [\textstyle \bigcup_{n\ge N} F\cap(\{n\}\times X)] \triangle F
$$
is compact, and so $(\bigcup_{n\ge N} F\cap(\{n\}\times X))^* = F^*$, and, similarly, for $G$.
\end{proof}

\begin{lemma}
$i\colon \mathcal{H}(X)\to i(\mathcal{H}(X))$ is open.
\end{lemma}

\begin{proof}
It suffices to show that $i$ is open at $e_{\mathcal{H}(X)}$.

Let $V\subseteq \mathcal{H}(X)$ be an open neighborhood of $e_{\mathcal{H}(X)}=\identity_X$. There is a finite clopen partition $\mathcal{U}$ of $X$ such that $N(\mathcal{U})\subseteq V$. We have to prove that $i(V)$ is a neighborhood of $i(e_{\mathcal{H}(X)})=\identity_{Z^*}$ in $i(\mathcal{H}(X))$. Put $\mathcal{K}= \{(\omega\times U)^*: U\in \mathcal{U}\}$. Then $\mathcal{K}$ is a finite clopen partition of $Z^*$. Hence $N(\mathcal{K})$ is a neigborhood of $\identity_{Z^*}$ in $\mathcal{H}(Z^*)$.

\begin{claim}
$i(N(\mathcal{U})) = N(\mathcal{K})\cap i(\mathcal{H}(X))$.
\end{claim}

Take an arbitrary $f\in N(\mathcal{U})$, and fix $U\in \mathcal{U}$. Then $f^{\omega}(\{n\}\times U) = \{n\}\times U$ for each $n < \omega$. Hence $f^\beta((\omega\times U)^*) = (\omega\times U)^*$, and so $i(f)((\omega\times U)^*) = (\omega\times U)^*$. We conclude that $i(f)\in N(\mathcal{K})$.

Conversely, assume that for some $f\in \mathcal{H}(X)$, $i(f)\in N(\mathcal{K})$. Again, fix $U\in \mathcal{U}$. Then $i(f)\in [(\omega\times U)^*, (\omega\times U)^*]$, and so by \ref{zaterdageen} (applied to both $f$ and $f^{-1}$), $\{n < \omega : f^\omega(\{n\}\times U) \not= \{n\}\times U\}$ is finite. Hence, $\mathcal{U}$ being finite, there exists $N$ such that for all $U\in \mathcal{U}$, $f^\omega(\{N\}\times U) = \{N\}\times U$, and so $f(U)=U$. We conclude that $f\in N(\mathcal{U})$, as desired.
\end{proof}


To conclude the proof of \ref{eerstestelling}, consider an arbitrary topological group $G$ of weight at most~${\mathfrak c}$. Then $G_{(\delta)}$ has weight at most ${\mathfrak c}$, and is non-archimedean, being a $P$-group. By \ref{megresh}, there is a zero-dimensional compact space $X$ such that $G_{(\delta)}$ is topologically isomorphic to a subgroup $H$ of $\mathcal{H}(X)$. Let $Z$ be as above. We know by the Lemmas that $i$ is an injective homomorphism, $i\colon \mathcal{H}(X)_{(\delta)}\to \mathcal{H}(Z^*)$ is continuous, and that $i\colon \mathcal{H}(X)\to i(\mathcal{H}(X))$ is open. Since $H$ is a $P$-space, $i{\restriction}H\colon H\to \mathcal{H}(Z^*)$ is continuous. Since $i$ is injective and open, it follows that $i{\restriction}H\colon H\to i(H)$ is open as well. We conclude that $i{\restriction}H\colon H\to i(H)$ is a topological isomorphism.

Observe that we actually proved a stronger result than stated: every $P$-group $G$ is topologically isomorphic to a subgroup of some $\mathcal{H}(X)$, where $X$ is a compact zero-dimensional $F$-space in which nonempty $G_\delta$'s have infinite interior. The assumption on the weight of $G$ in \ref{eerstestelling} was only used to conclude that the space $X$ can be chosen to be of weight ${\mathfrak c}$, making it a Parovi\v{c}enko space and hence homeomorphic to $\omega^*$ under {\ensuremath{\mathsf{CH}}}.

\section{Proof of \ref{tweedestelling}}
Since all Parovi\v{c}enko spaces are homeomorphic under {\ensuremath{\mathsf{CH}}}, one implication is trivial.

For the proof of the reverse implication, we use the Parovi\v{c}enko spaces in \cite{vm:14}. Assume that {\ensuremath{\mathsf{CH}}}\ fails. Our job is to find a contradiction. \addtocounter{claim}{-1}

\begin{claim}
There is a Parovi\v{c}enko space $S$ having the following properties:
\begin{enumerate}
\item for some $p\in S$, $\chi(p,S)=\omega_1$,
\item $S\setminus \{p\}$ is homeomorphic to an open subset of $\omega^*$.
\end{enumerate}
\end{claim}

\noindent This is the space considered in \cite[p.\ 539]{vm:14}.

\begin{claim}
There is a Parovi\v{c}enko space $T$ having the following properties:
\begin{enumerate}
\item all nonempty clopen subsets of $T$ are homeomorphic to $T$,
\item for each $x\in T$, $\chi(x,T) ={\mathfrak c}$.
\end{enumerate}
\end{claim}

\noindent Indeed, $T= (\omega\times 2^{\mathfrak c})^*$, the space considered in \cite[p.\ 540]{vm:14}. We only need to check (1) since (2) can be found in \cite[p.\ 540]{vm:14}. If $C\subseteq T$ is nonempty and clopen, then there is a noncompact clopen $D\subseteq \omega\times 2^{\mathfrak c}$ such that $\mathrm{cl}_{\beta T} D = D\cup C$, and so $D^* = C$. There are infinitely many $n$ such that $(\{n\}\times 2^{\mathfrak c})\cap D\not=\emptyset$. Since all nonempty clopen subsets of $2^{\mathfrak c}$ are homeomorphic to $2^{\mathfrak c}$, $D\approx \omega\times 2^{\mathfrak c}$, hence $D^* \approx T$.

\par\smallskip\noindent

Assume that $\mathcal{H}(T)$, $\mathcal{H}(S)$ and $\mathcal{H}(\omega^*)$ are pairwise isomorphic.

Since all nonempty clopen subsets of $\omega^*$ are homeomorphic to $\omega^*$, $T\approx \omega^*$ by \ref{rubin}. Hence for each $q\in \omega^*$, $\chi(q,\omega^*)={\mathfrak c}$.

Now consider $S$, and its point $p$. Since $S\setminus \{p\}$ is homeomorphic to an open subset of $\omega^*$, it follows from what we know that for each $r\in S\setminus \{p\}$, $\chi(r,S) = {\mathfrak c}$. Since $\chi(p,S)=\omega_1<{\mathfrak c}$, this implies that for each $f\in \mathcal{H}(S)$, $f(p)=p$. But then since $S$ is the Alexandroff one-point compactification of $S\setminus \{p\}$, $\mathcal{H}(S\setminus \{p\})$ and $\mathcal{H}(S)$ are isomorphic. It is also clear that $S\setminus \{p\}$, being homeomorphic to an open subset of $\omega^*$, has a base of homeomorphic copies of $\omega^*$. Hence both
$\langle{S\setminus \{p\}},{\mathcal{H}(S\setminus\{p\}}\rangle$ and $\langle{\omega^*},{\mathcal{H}(\omega^*)}\rangle$ are in $\mathfrak{R}$, while moreover $\mathcal{H}(S\setminus\{p\})$ and $\mathcal{H}(\omega^*)$ are isomorphic. So $S\setminus\{p\}$ and $\omega^*$ are homeomorphic by \ref{rubin}; this is a contradiction since one of them is compact, while the other one is not.

\section{Some topological properties of{\ensuremath{\mathcal{H}(\omega^*)}}}

We study the topological group {\ensuremath{\mathcal{H}(\omega^*)}}\ here.

\subsection{Cardinal functions}

We discuss the values of a few familiar cardinal functions on {\ensuremath{\mathcal{H}(\omega^*)}}; as an application, it will follow that {\ensuremath{\mathcal{H}(\omega^*)}}\ is not a $P$-group.

\begin{lemma}\label{hetlemma}
Let $B\subseteq \omega^*$ be clopen. and let $\mathcal{Z}$ be a finite clopen partition of $\omega^*$ such that $\bigcap_{Z\in \mathcal{Z}} [Z,Z] \subseteq [B,B]$. Then for each $Z\in \mathcal{Z}$, either $B\cap Z=\emptyset$ or $Z\subseteq B$.
\end{lemma}

\begin{proof}
Let $Z\in \mathcal{Z}$ be arbitrary, and assume that $B\cap Z\not =\emptyset$ and $Z\setminus B\not=\emptyset$.
Let $g \colon B\cap Z\to Z\setminus B$ be any homeomorphism, and define $f\in \mathcal{H}(\omega^*)$ by: $f$ restricts to the identity on $\omega^*\setminus Z$, restricts to $g$ on $B\cap Z$, and $g^{-1}$ on $Z\setminus B$.

Then $f\in \bigcap_{Z\in \mathcal{Z}} [Z,Z] \setminus [B,B]$, which is a contradiction.
\end{proof}

An
indexed family $\{(A_i^0, A_i^1): i \in I\}$ of pairs of disjoint subsets of a set $X$ is called
\emph{independent} provided that for all $\sigma\in [I]^{<\omega}$ and $\xi\in 2^\sigma$, the set $\bigcap_{i\in\sigma} A^{\xi(i)}_i$ is infinite.

\begin{lemma}\label{independent}
There is an independent family $\{(A_\alpha^0, A_\alpha^1): \alpha < {\mathfrak c}\}$ of pairs of disjoint subsets of $\mathcal{H}(\omega^*)$ such that
\begin{enumerate}
\item for each $\alpha < {\mathfrak c}$, $A_\alpha^0$ is a clopen subgroup of $\mathcal{H}(\omega^*)$, and $A_\alpha^1$ is its complement,
\item for each infinite $\sigma\subseteq {\mathfrak c}$, $\bigcap_{\alpha\in \sigma} A_\alpha^0$ is nowhere dense in $\mathcal{H}(\omega^*)$.
\end{enumerate}
\end{lemma}

\begin{proof}
Let $\mathcal{B}$ be a pairwise disjoint family nonempty clopen subsets of $\omega^*$, \cite[3.1.2(a)]{vanmill:betaomega} of size ${\mathfrak c}$, and enumerate it `faithfully' as $\{B_\alpha : \alpha < {\mathfrak c}\}$. We can split this family in two subfamilies, each of cardinality ${\mathfrak c}$. Hence we may assume that there is also a similarly indexed pairwise disjoint family nonempty clopen subsets $\{C_\alpha : \alpha < {\mathfrak c}\}$ of $\omega^*$, such that $\bigcup_{\alpha<{\mathfrak c}}B_\alpha\cap \bigcup_{\alpha < {\mathfrak c}} C_\alpha=\emptyset$.

For each $\alpha < {\mathfrak c}$, put $A_\alpha^0 = \hat B_\alpha^0$ (see \S\ref{prelimeen}).

For (1), pick arbitrary $\sigma\in [{\mathfrak c}]^{<\omega}$ and $\xi\in 2^\sigma$. Let $\tau = \{\alpha\in\sigma : \xi(\alpha) = 1\}$. Define an element $f\in {\ensuremath{\mathcal{H}(\omega^*)}}$, as follows. If $\alpha\in \tau$, then $f(B_\alpha) = C_\alpha$ and $f(C_\alpha)=B_\alpha$. Moreover, $f$ restricts to the identity on $\omega^*\setminus (\bigcup_{\alpha\in\tau}B_\alpha\cup C_\alpha)$.
Then $f\in \bigcap_{\alpha\in \sigma} A_\alpha^{\xi(\alpha)}$. It is easily seen that the intersection is in fact infinite.

For (2), let $\sigma\subseteq {\mathfrak c}$ be infinite. If $\bigcap_{\alpha\in \sigma} A^0_\alpha$ has nonempty interior, then it contains a neighborhood of the identity, being a subgroup. Hence it suffices to prove that this cannot happen. To this end, let $\mathcal{Z}$ be a finite clopen partition of $\omega^*$ by nonempty sets such that $\bigcap_{Z\in \mathcal{Z}}[Z,Z]$ is contained in $\bigcap_{\alpha\in \sigma} A^0_\alpha$. By \ref{hetlemma}, for each $\alpha\in\sigma$ and $Z\in \mathcal{Z}$, either $B_\alpha\cap Z=\emptyset$, or $Z\subseteq B_\alpha$. There exists $Z\in \mathcal{Z}$ such that $\{\alpha\in\sigma : B_\alpha\cap Z\not=\emptyset\}$ is infinite. Hence for distinct such $\delta$ and $\gamma$, $Z\subseteq B_\delta\cap B_\gamma$, which is a contradiction.
\end{proof}

\begin{corollary}\label{NotP}
$\mathcal{H}(\omega^*)$ is not a $P$-group.
\end{corollary}

We mentioned in \S\ref{introduction} that the cardinality of ${\ensuremath{\mathcal{H}(\omega^*)}}$ cannot be determined in {\ensuremath{\mathsf{ZFC}}. We see the same phenomenon in $\omega^*$. For example, both the minimum character and the minimum $\pi$-character of points of $\omega^*$ cannot be determined in {\ensuremath{\mathsf{ZFC}}, see \cite{vanmill:betaomega}; but of course, its cardinality can, it is $2^{\mathfrak c}$. Interestingly, {\ensuremath{\mathcal{H}(\omega^*)}}\ behaves similarly but differently.

\begin{proposition}\label{cardinal}
If $\phi\in \{w, \chi, \pi\chi, c,d,t\}$, then for each nonempty open $U$ in ${\ensuremath{\mathcal{H}(\omega^*)}}$, $\phi(U)={\mathfrak c}$.
\end{proposition}

\begin{proof}
Let $U\in \tau({\ensuremath{\mathcal{H}(\omega^*)}})\setminus\{\emptyset\}$. We may assume without loss of generality that $e\in U$.

 That $w({\ensuremath{\mathcal{H}(\omega^*)}}) \le{\mathfrak c}$ is clear, see \S\ref{prelimeen}, hence $w(U)\le {\mathfrak c}$.

Assume that $\{U_\gamma : \gamma < \kappa\}$, where $\kappa < {\mathfrak c}$, is a local base at $e$ in {\ensuremath{\mathcal{H}(\omega^*)}}.
Let $\{(A_\alpha^0, A_\alpha^1): \alpha < {\mathfrak c}\}$ be as in \ref{independent}. For each $\alpha <{\mathfrak c}$, there exists $\gamma(\alpha) < \kappa$ such that $U_{\gamma(\alpha)}\subseteq A^0_\alpha$. There is $\delta < \kappa$ such that the set $E=\{\alpha < {\mathfrak c} : \gamma(\alpha)=\delta\}$ is infinite. Hence $U_\delta\subseteq \bigcap_{\alpha \in E} A^0_\alpha$, which contradicts \ref{independent}(2). Hence $\chi(e,{\ensuremath{\mathcal{H}(\omega^*)}})={\mathfrak c}$, and so $\chi(U)={\mathfrak c}$.

By the same argument, $\pi\chi(U)={\mathfrak c}$ as well.

Let $\mathcal{Z}$ be a finite clopen partition of $\omega^*$ such that $N(\mathcal{Z})\subseteq U$. Pick a nonempty $Z\in \mathcal{Z}$.
Let $\mathcal{B}$ be a pairwise disjoint family nonempty clopen subsets of $Z$ of size ${\mathfrak c}$, \cite[3.1.2(a)]{vanmill:betaomega}. We may assume that there exists a nonempty clopen subset $C$ of $Z$ such that $C\cap \bigcup \mathcal{B}=\emptyset$. For each $B\in \mathcal{B}$, let $U_B = \{f\in \mathcal{H}(\omega^*) : (f(B)= C) \, \& \, (f(Z\setminus (B\cup C))=Z\setminus (B\cup C)\, \& \, (f{\restriction} (\omega^*\setminus Z)$ is the identity on $\omega^*\setminus Z)\}$.
It is easy to see that $\{U_B: B\in \mathcal{B}\}$ is pairwise disjoint clopen family of subsets of $N(\mathcal{Z})$ of cardinality ${\mathfrak c}$.

 Hence ${\mathfrak c} \le c(U)\le d(U) \le w(U)\le{\mathfrak c}$, and so we get $c=d=w={\mathfrak c}$.

Now we deal with the tightness of {\ensuremath{\mathcal{H}(\omega^*)}}. This is more tricky. By \cite[3.3.4]{vanmill:betaomega}, there are an open $F_\sigma$-subset $V$ of $\omega^*$ and a point $p\in \overline{V}\setminus V$ such that for every $F\in [V]^{<{\mathfrak c}}$, $p\not\in \overline{F}$ (the proof of this result is based on a highly nontrivial matrix of clopen subsets of $\omega^*$ due to Kunen~\cite{Kunen78}). Let $\{\mathcal{Z}_\alpha : \alpha < {\mathfrak c}\}$ enumerate all finite clopen partitions of $\omega^*$. For every $\alpha < {\mathfrak c}$, there is a unique $Z_\alpha\in \mathcal{Z}_\alpha$ that contains~$p$ and hence intersects $V$.
Let $f_\alpha\in{\ensuremath{\mathcal{H}(\omega^*)}}$ be such that $f_\alpha(Z_\alpha)=Z_\alpha$, $f_\alpha(p)\in V$ (use that nonempty
clopen subsets of $\omega^*$ are homeomorphic to $\omega^*$), and $f_\alpha$ restricts to the identity on $\omega^*\setminus Z_\alpha$. Then $f_\alpha\in N(\mathcal{Z}_\alpha)\setminus \{e\}$. Hence $e\in \overline{\{f_\alpha : \alpha < {\mathfrak c}\}}$. Let $E\in [{\mathfrak c}]^{<{\mathfrak c}}$, and assume that $e\in \overline{\{f_\alpha : \alpha \in E\}}$.

\begin{claim}
$p\in \overline{\{f_\alpha(p): \alpha\in E\}}$.
\end{claim}

Indeed, let $W$ be an arbitrary clopen neighborhood of $p$, and put $\tilde W = \{g\in {\ensuremath{\mathcal{H}(\omega^*)}}: g(W)=W\}$. Then $\tilde W$ is a clopen neighborhood of $e$ in {\ensuremath{\mathcal{H}(\omega^*)}}, hence there exists $\alpha\in E$ such that $f_\alpha\in \tilde W$, and so $f_\alpha(p)\in W$.

 \par\medskip\noindent

Hence we arrived at a contradiction since $\{f_\alpha(p): \alpha\in E\}\in [V]^{<{\mathfrak c}}$ and so, by the properties of $p$, $p\not\in \overline{\{f_\alpha(p): \alpha\in E\}}$.
\end{proof}

\begin{corollary}\label{easylemma}
{\ensuremath{\mathcal{H}(\omega^*)}}\ is crowded.
\end{corollary}

\subsection{Universality properties of {\ensuremath{\mathcal{H}(\omega^*)}}.}
In light of \ref{eerstecor}, it is a natural question whether $\mathcal{H}(\omega^*)$ is universal for all non-archimedean groups of small weight. It is not, as we will now show.

\begin{proposition}\label{ditisem}
If $X$ is a crowded compact zero-dimensional $F$-space in which nonempty $G_{\delta}$'s have nonempty interior, then all countable subsets of $\mathcal{H}(X)$ are discrete (and hence closed and discrete).
\end{proposition}

\begin{proof}
Assume that $\{f_n : n < \omega\}\subseteq \mathcal{H}(X)\setminus \{e\}$. Pick a nonempty clopen subset $C_0^0$ in $X$ such that $C_0^0\cap f_0(C^0_0)=\emptyset$, and a point $x_0\in C^0_0$. By recursion on $1\le n < \omega$, we will construct $x_n\in X$, and for each $m \le n$ a clopen neighborhood $C^n_m$ of $x_m$, such that
\begin{enumerate}
\item $\{C^n_m : m\le n\}\cup \{f_m(C_m^n) : m \le n\}$ has cardinality $2(n{+}1)$ and is pairwise disjoint,
\item for $m \le n{-}1$, $C_m^n \subseteq C_m^{n-1}$.
\end{enumerate}
Assume that for certain $n < \omega$, we constructed $\{x_m : m\le n\}$ and the family $\{C_m^n: m\le n\}$. By (1), $F=\{x_m : m \le n\}\cup \{f_m(x_m): m \le n\}$ has cardinality $2(n{+}1)$. Put $G=F\cup f_{n+1}^{-1}(F)$, and let $U$ be a nonempty clopen subset of $X$ such that $U\cap f_{n+1}(U)=\emptyset$. Since $X$ is crowded, $U$ is infinite and $G$ is finite, we may pick $x_{n+1}\in U\setminus G$. Hence $H=\{x_m : m \le n{+1}\}\cup \{f_m(x_m): m \le n{+}1\}$ has cardinality $2(n{+}2)$. For each $m\le n{+}1$, let $S_m$ and $T_m$ be clopen neighborhoods of $x_m$ respectively $f_{m}(x_m)$ such that $T_m=f_m(S_m)$, the family $\{S_m : m\le n{+}1\}\cup \{T_m : m\le n{+}1\}$ has cardinality $2(n{+}2)$ and is pairwise disjoint. Now for each $m \le n$, put $C^{n+1}_m = C^n_m\cap S_m$, and let $C^{n+1}_{n+1} = S_{n+1}\cap U$. Then our choices are clearly as desired, which completes the recursion.

Now for each $n$, the intersection of the collection $\{C^k_n : k \ge n\}$ contains $x_n$. Hence, by assumption, we there is a nonempty clopen $E_n \subseteq\bigcap_{k\ge n} C^k_n$. Then $\bigcup_n E_n \cap \bigcup_n f_n(E_n)=\emptyset$ by (1), and hence, since $X$ is an $F$-space, there is a clopen subset $E$ of $X$ such that $\bigcup_n E_n\subseteq E\subseteq X\setminus \bigcup_n f_n(E_n)$. Now $U= \{g\in \mathcal{H}(X) : g(E)=E\}$ is an open neighborhood of $e$ in $\mathcal{H}(X)$ not containing any $f_n$.
\end{proof}

This means that the familiar topological group $2^\omega$, which is non-archimedean, is not topologically isomorphic to a subgroup of $\mathcal{H}(\omega^*)$.

The assumption on zero-dimensionality in \ref{ditisem}, is in fact superfluous: a slight adaptation of the proof gives the result for all compact crowded $F$-spaces in which nonempty $G_\delta$'s have nonempty interior.

We do not know whether \ref{ditisem} can be generalized for higher cardinals, even for $\omega^*$.
By \ref{cardinal}, {\ensuremath{\mathcal{H}(\omega^*)}}\ is crowded and has density ${\mathfrak c}$, hence there are subsets of ${\ensuremath{\mathcal{H}(\omega^*)}}$ of size ${\mathfrak c}$ that are not discrete. Hence our question is for cardinals between $\omega$ and ${\mathfrak c}$.
An inspection of the proof of \ref{ditisem} will reveal that for subsets say of cardinality $\omega_1$, one can run for example into a Hausdorff gap (a nightmare for $\omega^*$ af\i cionados), or perhaps ultrafilters with character $\omega_1$ (another nightmare).

\begin{question}
Is there in {\ensuremath{\mathsf{ZFC}}} a subset of {\ensuremath{\mathcal{H}(\omega^*)}} of size $\omega_1$ that is not closed?
\end{question}

It is also a natural question whether {\ensuremath{\mathcal{H}(\omega^*)}}\ contains all `small' $P$-groups in {\ensuremath{\mathsf{ZFC}}. It is not universal for `small' $P$-groups since it is not a $P$-group itself, see \ref{NotP}. The question cannot be answered in {\ensuremath{\mathsf{ZFC}}. As we showed, it does under {\ensuremath{\mathsf{CH}}}. But it does not in Shelah's model where {\ensuremath{\mathcal{H}(\omega^*)}}\ has cardinality ${\mathfrak c}$ (see \S{\ref{introduction}}), simply because the group $2^{\mathfrak c}_{(\delta)}$, which has weight ${\mathfrak c}$, is too large.

\par\bigskip\noindent

\noindent {\bf{Acknowledgements:}} I am indebted to Sergey Antonyan, Alan Dow, Klaas Pieter Hart, Michael Megrelishvili and Vladimir Uspenskiy for interesting discussions and useful comments.


\def\cprime{$'$}
\makeatletter \renewcommand{\@biblabel}[1]{\hfill[#1]}\makeatother

\end{document}